\documentclass[pdflatex,sn-mathphys-num]{sn-jnl}

\usepackage{enumitem}

\usepackage{hyperref}
\usepackage{amsthm}
\usepackage{chngcntr}
\counterwithin*{equation}{section}
\counterwithin*{equation}{subsection}
\usepackage{stmaryrd}
\usepackage[all]{xy}
\usepackage{tikz-cd}
\usepackage{stmaryrd}
\usepackage[utf8]{inputenc}
\usepackage{csquotes}
\usepackage{cooltooltips}
\usepackage[english]{babel}

\usepackage{amssymb}
\usepackage{mathtools}

\usepackage[T1]{fontenc}
\usepackage{textcomp}

\usepackage{color}

\newtheorem{theorem}{Theorem}[section]
\newtheorem{corollary}[theorem]{Corollary}
\newtheorem{lemma}[theorem]{Lemma}
\newtheorem{definition}[theorem]{Definition}
\newtheorem{definitionlemma}[theorem]{Definition-Lemma}
\newtheorem{proposition}[theorem]{Proposition}
\newtheorem{remark}[theorem]{Remark}

\let\originalmiddle=\middle
\renewcommand{\middle}[1]{\mathrel{}\originalmiddle#1\mathrel{}}

\newcommand{\ee}[2]{{\sf \underline{e}}_{\scalebox{0.5}{$#1,#2$}}}

\newcommand{\e}[1]{{ \sf e}_{\scalebox{0.5}{$#1$}}}
\newcommand{\hq}[4]{\ensuremath{h_{#1,#2}^{#3,#4}}}
\newcommand{\hh}[5]{\ensuremath{h_{#1,#2,#3}^{#4,#5}}}
\newcommand{\h}[3]{\ensuremath{h_{#1}^{#2,#3}}}

\newcommand\restr[2]{{
    \left.\kern-\nulldelimiterspace 
    #1 
    \vphantom{\big|} 
    \right|_{#2} 
}}

\DeclareMathOperator{\Aut}{Aut}
\DeclareMathOperator{\Hom}{Hom}

\DeclareMathOperator{\End}{End}

\begin{document}

\title{Prime-Power Multiplicative Modular Monoid Morphisms}
\author[1]{Joseph Atalaye\footnote{Joseph Atalaye ORCID: \href{https://orcid.org/0009-0005-7329-2889}{0009-0005-7329-2889}}}
\email{26828146@sun.ac.za}
\equalcont{These authors contributed equally to this work.}

\author*[1]{Liam Baker\footnote{Liam Baker ORCID: \href{https://orcid.org/0000-0003-2728-3963}{0000-0003-2728-3963}}~}
\email{liambaker@sun.ac.za}
\equalcont{These authors contributed equally to this work.}

\author[2]{Sophie Marques\footnote{Sophie Marques ORCID: \href{https://orcid.org/0000-0001-8700-3375}{0000-0001-8700-3375}}}
\email{smarques@sun.ac.za}
\equalcont{These authors contributed equally to this work.}

\affil*[1]{\orgdiv{Department of Mathematical Sciences}, \orgname{Stellenbosch University}, \orgaddress{\city{Stellenbosch}, \postcode{7600}, \country{South Africa}}}
\affil[2]{\orgdiv{Centre of Mathematics}, \orgname{University of Minho}, \orgdiv{\city{Braga}, \country{Portugal}}}


\abstract{
We undertake a systematic study of multiplicative monoid morphisms between prime-power modular monoids $(\mathbb{Z}/p^e\mathbb{Z},\cdot)$ and $(\mathbb{Z}/q^f\mathbb{Z},\cdot)$, for arbitrary primes $p,q$ and integers $e,f\geq 1$.
Our approach proceeds in three stages.
We first determine the hom-set $\Hom(U_{p^e},U_{q^f})$ of group morphisms between the corresponding unit groups and exhibit explicit isomorphisms describing its group structure.
We then identify the subset $\mathcal{H}_{p^e,q^f,v}$ of unit group morphisms that extend to full monoid morphisms on $(\mathbb{Z}/p^e\mathbb{Z},\cdot)$, characterizing it via explicit divisibility conditions and describing its algebraic structure.
Finally, we give a complete description of $\Hom_{\mathbf{Mon}}((\mathbb{Z}/p^e\mathbb{Z},\cdot), (\mathbb{Z}/q^f\mathbb{Z},\cdot))$ and establish an explicit monoid isomorphism between this hom-set and a naturally defined monoid built from unit groups of lower-level modular rings.
}

\maketitle

\setcounter{tocdepth}{3}

\section*{Introduction}

The prime-power multiplicative modular monoids $(\mathbb{Z}/p^e\mathbb{Z},\cdot)$
occupy a fundamental position in algebra and number theory. Despite their natural
importance, the morphism theory of such monoids remains largely unexplored compared
with the well-developed theories of additive group homomorphisms and ring
homomorphisms. The present paper undertakes a systematic and complete determination
of the hom-set
\[
\Hom_{\mathbf{Mon}}\bigl((\mathbb{Z}/p^e\mathbb{Z},\cdot),
(\mathbb{Z}/q^f\mathbb{Z},\cdot)\bigr)
\]
for arbitrary primes $p,q$ and integers $e,f\geq 1$, and investigates the algebraic
structure it carries. Four independent motivations justify this study.

\medskip
\noindent\textbf{Ring deformations and multiplicative automorphisms.}
As observed in \cite{sophieLee}, every multiplicative automorphism
$\phi\in\Aut(\mathbb{Z}/p^e\mathbb{Z},\cdot)$ induces a new ring structure on
$\mathbb{Z}/p^e\mathbb{Z}$ via the twisted addition
\[
a+_\phi b \;=\; \phi^{-1}(\phi(a)+\phi(b)),
\quad a,b\in\mathbb{Z}/p^e\mathbb{Z},
\]
yielding a ring $(\mathbb{Z}/p^e\mathbb{Z},+_\phi,\cdot)$ isomorphic to the
original through $\phi$. Moreover, $\phi$ induces an action
$a\cdot_\phi b = \phi(a)b$ of the deformed ring on the original additive group,
making $(\mathbb{Z}/p^e\mathbb{Z},+,\cdot)$ an algebra over
$(\mathbb{Z}/p^e\mathbb{Z},+_\phi,\cdot)$. More generally, a multiplicative
monoid morphism
$\varphi:(\mathbb{Z}/p^e\mathbb{Z},\cdot)\to(\mathbb{Z}/q^f\mathbb{Z},\cdot)$
encodes a compatibility between the multiplicative structures on both sides, and
a precise description of
$\Hom_{\mathbf{Mon}}\bigl((\mathbb{Z}/p^e\mathbb{Z},\cdot),
(\mathbb{Z}/q^f\mathbb{Z},\cdot)\bigr)$
constitutes a fundamental step toward a systematic study of such deformations.
Automorphisms of monoids have been investigated in various contexts; see
\cite{putcha1983automorphism,can2021incidence,
szechtman2005automorphism,atalaye2026automorphism}.

\medskip
\noindent\textbf{Monoid actions and their classification.}
A left action of a monoid $M$ on a set $X$ is a map $M\times X\to X$ satisfying
$\mathbf{1}_M x = x$ and $(m_1m_2)x=m_1(m_2x)$ for all $x\in X$ and
$m_1,m_2\in M$ \cite[p.~8]{steinberg2016representation}, or equivalently a
monoid homomorphism $\rho:M\to\End(X)$. Thus classifying all $M$-actions on
$X$ amounts precisely to describing $\Hom_{\mathbf{Mon}}(M,\End(X))$, placing
the computation of monoid homomorphisms at the heart of the theory. Since
$(\mathbb{Z}/p^e\mathbb{Z},\cdot)$ possesses an absorbing element $0$, every
action of this monoid reflects the interplay between its invertible and
noninvertible elements: units correspond to reversible dynamics, idempotents to
stabilization, and nilpotents to collapse toward absorbing states. Understanding
\[
\Hom_{\mathbf{Mon}}\bigl((\mathbb{Z}/p^e\mathbb{Z},\cdot),
(\mathbb{Z}/q^f\mathbb{Z},\cdot)\bigr)
\quad\text{and}\quad
\End_{\mathbf{Mon}}(\mathbb{Z}/p^e\mathbb{Z},\cdot)
\]
thus provides a foundation for classifying actions and representation-theoretic
structures associated with modular arithmetic.

\medskip
\noindent\textbf{Semidirect products of modular monoids.}
Given monoids $S$ and $T$ together with an action $\rho:T\to\End(S)$, the
semidirect product $S\rtimes_\rho T$ is defined on $S\times T$ by
\[
(s_1,t_1)(s_2,t_2)=\bigl(s_1\,\rho(t_1)(s_2),\;t_1 t_2\bigr),
\]
a construction developed in
\cite{usenko1982semidirect,preston1986semidirect,
blanchet1996semidirect,saito1989orthodox}.
A natural instance is the semidirect product of $(\mathbb{Z}/q^f\mathbb{Z},+)$
with $(\mathbb{Z}/p^e\mathbb{Z},\times)$, which requires a monoid homomorphism
\[
\varphi:(\mathbb{Z}/p^e\mathbb{Z},\times)
\longrightarrow \End(\mathbb{Z}/q^f\mathbb{Z},+).
\]
Since $(\mathbb{Z}/q^f\mathbb{Z},+)$ is cyclic, every additive endomorphism is
multiplication by some element of $\mathbb{Z}/q^f\mathbb{Z}$, yielding a
canonical identification
$\End(\mathbb{Z}/q^f\mathbb{Z},+)\cong(\mathbb{Z}/q^f\mathbb{Z},\times)$.
Hence the study of such semidirect products reduces precisely to describing
$\Hom_{\mathbf{Mon}}\bigl((\mathbb{Z}/p^e\mathbb{Z},\times),
(\mathbb{Z}/q^f\mathbb{Z},\times)\bigr)$,
which simultaneously classifies all admissible actions and determines which
semidirect products of additive and multiplicative modular structures exist and
how they behave.

\medskip
\noindent\textbf{Automorphisms of direct products.}
A further motivation arises from our recent work \cite{atalaye2026automorphism}
on the automorphism group of direct products
$\prod_{i=1}^{n}(\mathbb{Z}/p_i^{e_i}\mathbb{Z},\cdot)$.
Every endomorphism of such a product is represented by a matrix
$(\alpha_{ij})$ with entries
\[
\alpha_{ij}\in
\Hom\!\bigl(\mathbb{Z}/p_j^{e_j}\mathbb{Z},\,\mathbb{Z}/p_i^{e_i}\mathbb{Z}\bigr)
\]
subject to explicit compatibility conditions. A precise description of the
individual hom-sets
$\Hom\bigl(\mathbb{Z}/p_j^{e_j}\mathbb{Z},\mathbb{Z}/p_i^{e_i}\mathbb{Z}\bigr)$
is therefore indispensable for determining the endomorphism monoid and,
ultimately, the automorphism group of such products.

\medskip
These four motivations converge on the explicit computation of
$\Hom_{\mathbf{Mon}}\bigl((\mathbb{Z}/p^e\mathbb{Z},\cdot),
(\mathbb{Z}/q^f\mathbb{Z},\cdot)\bigr)$,
which to the best of our knowledge has not been carried out in full generality
in the existing literature. Our approach proceeds in three stages:
Section~\ref{unithom} computes group morphisms $U_{p^e}\to U_{q^f}$ and their
algebraic structure; Section~\ref{lifunithom} determines which unit morphisms
extend to full monoid morphisms; and Section~\ref{modhom} synthesizes these
results into a complete classification of
$\Hom_{\mathbf{Mon}}\bigl((\mathbb{Z}/p^e\mathbb{Z},\cdot),
(\mathbb{Z}/q^f\mathbb{Z},\cdot)\bigr)$
and investigates its global algebraic structure.

\medskip
\textbf{Notation:} for the duration of the paper we consider:
\begin{itemize}
    \item $p$ and $q$ to be prime,
    \item $e$ and $f$ to be positive integers, and
    \item when the unit group $U_{p^e}$ is cyclic, we fix $a \in \mathbb{Z}$ such that $[a]_{p^e}$ generates $U_{p^e}$, and similarly for $[c]_{q^f}$ generating $U_{q^f}$ when it is cyclic.
\end{itemize}

\section{Unit morphisms}\label{unithom}

The study of multiplicative monoid morphisms between prime-power 
modular monoids rests fundamentally on understanding group 
homomorphisms between the corresponding unit groups. We begin by 
recalling a basic compatibility between generators under reduction. 
Given $0\leq f\leq e$ and $a\in\mathbb{Z}$, it is clear that 
$[a]_{p^e}\in U_{p^e}$ implies $[a]_{p^f}\in U_{p^f}$. Moreover:
\begin{itemize}
    \item when $p$ is odd, if $[a]_{p^e}$ generates $U_{p^e}$ 
    then $[a]_{p^f}$ generates $U_{p^f}$;
    \item when $p=2$ and $e\geq 3$, if $\{[a]_{2^e},[b]_{2^e}\}$ 
    generates $U_{2^e}$ then $\{[a]_{2^f},[b]_{2^f}\}$ generates 
    $U_{2^f}$.
\end{itemize}

\subsection{Computation of the hom-set of unit morphisms}
\label{setunithom}

We now describe explicitly all group homomorphisms between the unit groups $U_{p^e}$ and $U_{q^f}$.
The structure of these morphisms depends on the parities of $p$ and $q$ and on the value of $e$.
The proof of the following lemma is left to the reader, but it follows from the group structure of $U_{p^e}$.

\begin{lemma}\label{Lem411}
\begin{enumerate}
    \item If $p$ is even and $e\in\{1,2\}$, or $p$ is odd. 
    Then
    \[
    \Hom(U_{p^e},U_{q^f})
    =\left\{
    \hh{a}{c}{t}{p^e}{q^f}
    :t\in\mathbb{Z},\;
    \frac{\varphi(q^f)}{\gcd(\varphi(q^f),\varphi(p^e))}\mid t
    \right\},
    \]
    where $\hh{a}{c}{t}{p^e}{q^f}:U_{p^e}\to U_{q^f}$ is the 
    well-defined group homomorphism sending $[a]_{p^e}^u$ to 
    $[c]_{q^f}^{tu}$ for all $u\in\mathbb{N}$.

    \item If $p=q=2$ and $e\geq 3$, then
    \[
    \Hom(U_{2^e},U_{2^f})
    =\bigl\{\h{t}{2^e}{2^f}:t=(t_1,t_2,t_3,t_4)\in\mathbb{Z}^4\bigr\},
    \]
    where $\h{t}{2^e}{2^f}:U_{2^e}\to U_{2^f}$ is the well-defined 
    group homomorphism sending
    \[
    [-1]_{2^e}\longmapsto [-1]_{2^f}^{t_1}[5]_{2^f}^{2^{f-3}t_2}
    \quad\text{and}\quad
    [5]_{2^e}\longmapsto [-1]_{2^f}^{t_3}[5]_{2^f}^{\nu_f(e)t_4},
    \]
    with
    \[
    \nu_f(e)=
    \begin{cases}
    2^{f-e} & \text{if }f\geq e,\\
    1 & \text{otherwise.}
    \end{cases}
    \]

    \item If $p=2$, $e\geq 3$, and $q$ is odd.
    Then
    \[
    \Hom(U_{2^e},U_{q^f})
    =\left\{
    \hq{c}{t}{2^e}{q^f}
    :t=(t_1,t_2)\in\mathbb{Z}^2,\;
    \varphi(q^f)\mid 2t_1,\;
    \frac{\varphi(q^f)}{\gcd(\varphi(q^f),2^{e-2})}\mid t_2
    \right\},
    \]
    where $\hq{c}{t}{2^e}{q^f}:U_{2^e}\to U_{q^f}$ is the 
    well-defined group homomorphism sending $[-1]_{2^e}$ to 
    $[c]_{q^f}^{t_1}$ and $[5]_{2^e}$ to $[c]_{q^f}^{t_2}$.
\end{enumerate}
\end{lemma}

\begin{remark}\label{Rem412}
Specializing statement~(1) to $p=q$ and $e=f$, the divisibility 
condition on $t$ reduces to $1\mid t$, which is vacuous. Hence,
\(\End(U_{p^e},\cdot)
=\bigl\{\hh{a}{a}{t}{p^e}{p^e}\mid t\in\mathbb{Z}\bigr\}.
\)
\end{remark}

\subsection{Group structure of the hom-set of unit morphisms}
\label{Struchom}

We now identify the algebraic structure of $\Hom(U_{p^e},U_{q^f})$ 
endowed with the pointwise product: for $\varphi,\psi\in 
\Hom(U_{p^e},U_{q^f})$, the product $\varphi\psi$ is defined by
\[
(\varphi\psi)(x)=\varphi(x)\psi(x)
\quad\text{for all }x\in U_{p^e},
\]
where the multiplication on the right-hand side takes place in $U_{q^f}$.
Since $U_{q^f}$ is abelian, this operation is again a group homomorphism, making $\Hom(U_{p^e},U_{q^f})$ itself a group. 
The following proposition exhibits explicit isomorphisms with familiar abelian groups.
We are leaving the proof to the reader since it is a routine group homomorphic theoretic check.

\begin{proposition}\label{Prop313}
\begin{enumerate}
    \item If $p$ is even and $e\in\{1,2\}$, or $p$ is odd 
    and set $d=\gcd(\varphi(q^f),\varphi(p^e))$. The map
    \[
    H_a^{p^e,q^f}:
    \bigl(\Hom(U_{p^e},U_{q^f}),\cdot\bigr)
    \longrightarrow
    (\mathbb{Z}/d\mathbb{Z},+),
    \qquad
    \hh{a}{c}{t}{p^e}{q^f}
    \longmapsto
    \left[\frac{td}{\varphi(q^f)}\right]_d,
    \]
    is a well-defined group isomorphism. In particular,
    \[
    \Hom(U_{p^e},U_{q^f})\cong\mathbb{Z}/\gcd(\varphi(q^f),\varphi(p^e))\mathbb{Z}.
    \]

    \item If $p=q=2$ and $e\geq 3$, the map
    \[
    H^{2^e,2^f}:
    \bigl(\Hom(U_{2^e},U_{2^f}),\cdot\bigr)
    \longrightarrow
    \bigl((\mathbb{Z}/2\mathbb{Z})^3
    \times\mathbb{Z}/2^{\min(e,f)-2}\mathbb{Z},+\bigr),
    \]
    \[
    \h{t}{2^e}{2^f}
    \longmapsto
    \bigl([t_1]_2,[t_2]_2,[t_3]_2,[t_4]_{2^{\min(e,f)-2}}\bigr),
    \]
    is a well-defined group isomorphism. In particular,
    \[
    \Hom(U_{2^e},U_{2^f})
    \cong(\mathbb{Z}/2\mathbb{Z})^3
    \times\mathbb{Z}/2^{\min(e,f)-2}\mathbb{Z}.
    \]

    \item If $p=2$, $e\geq 3$, and $q$ is odd, let 
    $a\in\mathbb{Z}$ with $o([c]_{q^f})=\varphi(q^f)$, and set
    $d=\gcd(\varphi(q^f),2^{e-2})$. The map
    \( 
    H_a^{2^e,q^f}:
    \bigl(\Hom(U_{2^e},U_{q^f}),\cdot\bigr)
    \longrightarrow
    (\mathbb{Z}/2\mathbb{Z}\times\mathbb{Z}/d\mathbb{Z},+),\)
 \[
    \hq{c}{t}{2^e}{q^f}
    \longmapsto
    \left(\left[\frac{2t_1}{\varphi(q^f)}\right]_2,
    \left[\frac{dt_2}{\varphi(q^f)}\right]_d\right),
    \]
    is a well-defined group isomorphism. In particular,
    \[
    \Hom(U_{2^e},U_{q^f})
    \cong\mathbb{Z}/2\mathbb{Z}
    \times\mathbb{Z}/\gcd(\varphi(q^f),2^{e-2})\mathbb{Z}.
    \]
\end{enumerate}
\end{proposition}

\section{Liftable unit morphisms}\label{lifunithom}

We define a distinguished subset of $\Hom(U_{p^e}, U_{q^f})$
consisting of those morphisms which, as we establish in
Section~\ref{modhom}, extend to monoid morphisms
$(\mathbb{Z}/p^e\mathbb{Z},\cdot)\to(\mathbb{Z}/q^f\mathbb{Z},\cdot)$. 

\subsection{Computation of the set of liftable unit morphisms}
\label{setlifunithom}

\begin{definition}\label{Definition421}
Let  $v \in \mathbb{N}^*$ with $v \leq f$. For $\ell \in \{p, q\}$ and integers $m \geq n \geq 0$,
denote by
\[
\pi^\ell_{m,n} : U_{\ell^m} \longrightarrow U_{\ell^n},
\qquad [r]_{\ell^m} \longmapsto [r]_{\ell^n},
\]
the canonical surjective group homomorphism, with the convention
$U_{\ell^0} := \{1\}$. We define $\mathcal{H}_{p^e,q^f,v}$ to be
the set of all group homomorphisms
$\phi : (U_{p^e}, \cdot) \to (U_{q^f}, \cdot)$ such that, for
every $0 \leq u \leq \min\!\left(e, \lfloor f/v \rfloor\right)$,
the map
\[
\phi_{u,v} : U_{p^{e-u}} \longrightarrow U_{q^{f-uv}},
\qquad [r]_{p^{e-u}} \longmapsto \pi^q_{f,f-uv}(\phi([r]_{p^e})),
\]
is a well-defined group homomorphism, i.e. independent of the
choice of lift $[r]_{p^e}$ of $[r]_{p^{e-u}}$ along
$\pi^p_{e,e-u}$. When it exists, $\phi_{u,v}$ is unique by
surjectivity of $\pi^p_{e,e-u}$, and is called the
\emph{$(u,v)$-induced morphism} associated with $\phi$.
Equivalently, $\phi \in \mathcal{H}_{p^e,q^f,v}$ if and only
if the diagram
\begin{equation}
    \label{diag:phi_uv}
    \xymatrix{
    U_{p^e} \ar[d]_{\pi^p_{e,e-u}} \ar[r]^{\phi}
    & U_{q^f} \ar[d]^{\pi^q_{f,f-uv}} \\
    U_{p^{e-u}} \ar[r]^{\phi_{u,v}}
    & U_{q^{f-uv}}
    }
\end{equation}
commutes for all such $u$. When $e = u$ and $f = uv$
simultaneously, both groups collapse to $\{1\}$ and
$\phi_{u,v}$ is the trivial morphism.
\end{definition}

\begin{corollary}\label{Cor424}
Let $0\leq u\leq e$. The map
\(\Theta_{p^e} :
(\mathcal{H}_{p^e,p^e,1},\circ)
\longrightarrow
(\mathcal{H}_{p^{e-u},p^{e-u},1},\circ),
\ 
\phi\longmapsto\phi_{u,1}
\)
is a well-defined monoid homomorphism.
Moreover, for all
$u,u'\in\mathbb{N}$ with $0\leq u'\leq u\leq e$:
\[
(\phi_{u,1})_{u',1} = \phi_{u',1}.
\]
\end{corollary}

\begin{proof}
\textbf{Well-definedness.}
We must show that $\phi_{u,1}\in\mathcal{H}_{p^{e-u},p^{e-u},1}$
for every $\phi\in\mathcal{H}_{p^e,p^e,1}$, and that
$(\phi_{u,1})_{u',1}=\phi_{u',1}$ for all $0\leq u'\leq u\leq e$.
Consider the diagram:
\[
\begin{tikzcd}[row sep=large, column sep=large]
U_{p^e}
  \arrow[r, "\pi^p_{e,\,e-u}"]
  \arrow[rr, bend left=35, "\pi^p_{e,\,e-u'}"]
  \arrow[d, "\phi"']
& U_{p^{e-u}}
  \arrow[r, "\pi^p_{e-u,\,e-u'}"]
  \arrow[d, "\phi_{u,1}"']
& U_{p^{e-u'}}
  \arrow[d, "(\phi_{u,1})_{u',1}"] \\
U_{p^e}
  \arrow[r, "\pi^p_{e,\,e-u}"']
  \arrow[rr, bend right=35, "\pi^p_{e,\,e-u'}"']
& U_{p^{e-u}}
  \arrow[r, "\pi^p_{e-u,\,e-u'}"']
& U_{p^{e-u'}}
\end{tikzcd}
\]
The left square commutes by definition of $\phi_{u,1}$
(Definition~\ref{Definition421}), and the right square commutes
by definition of $(\phi_{u,1})_{u',1}$. Hence the outer rectangle
commutes. Since
\[
\pi^p_{e,e-u'} = \pi^p_{e-u,e-u'}\circ\pi^p_{e,e-u},
\]
the outer rectangle says precisely that
$(\phi_{u,1})_{u',1}\circ\pi^p_{e,e-u'}=\pi^p_{e,e-u'}\circ\phi$.
By uniqueness of the induced morphism (Definition~\ref{Definition421},
using surjectivity of $\pi^p_{e,e-u'}$), this forces
$(\phi_{u,1})_{u',1}=\phi_{u',1}$, establishing well-definedness.

\medskip
\noindent\textbf{Monoid homomorphism.}
Let $\phi,\psi\in\mathcal{H}_{p^e,p^e,1}$. Consider the diagram:
\[
\begin{tikzcd}[row sep=large, column sep=large]
U_{p^e}
  \arrow[r, "\phi"]
  \arrow[d, "\pi^p_{e,\,e-u}"']
  \arrow[rr, bend left=25, "\psi\circ\phi"]
& U_{p^e}
  \arrow[r, "\psi"]
  \arrow[d, "\pi^p_{e,\,e-u}"']
& U_{p^e}
  \arrow[d, "\pi^p_{e,\,e-u}"] \\
U_{p^{e-u}}
  \arrow[r, "\phi_{u,1}"']
  \arrow[rr, bend right=25, "\psi_{u,1}\circ\phi_{u,1}"']
& U_{p^{e-u}}
  \arrow[r, "\psi_{u,1}"']
& U_{p^{e-u}}
\end{tikzcd}
\]
The left square commutes by definition of $\phi_{u,1}$, and the
right square commutes by definition of $\psi_{u,1}$. Hence the
outer rectangle commutes, giving
\[
(\psi_{u,1}\circ\phi_{u,1})\circ\pi^p_{e,e-u}
=\pi^p_{e,e-u}\circ(\psi\circ\phi).
\]
By uniqueness of $(\psi\circ\phi)_{u,1}$, we conclude
\[
(\psi\circ\phi)_{u,1} = \psi_{u,1}\circ\phi_{u,1},
\]
i.e. $\Theta_{p^e}(\psi\circ\phi)
=\Theta_{p^e}(\psi)\circ\Theta_{p^e}(\phi)$.
Finally, $\Theta_{p^e}(\mathrm{id}_{U_{p^e}})
=(\mathrm{id}_{U_{p^e}})_{u,1}=\mathrm{id}_{U_{p^{e-u}}}$,
so $\Theta_{p^e}$ preserves the identity. Hence $\Theta_{p^e}$
is a monoid homomorphism.
\end{proof}

From this section we denote by \(K^{p}_{e, e-u}:=[1]_{p^e} +\langle p^{e-u}\rangle_{p^e}\) (resp. \(K^{q}_{f, f-uv}:=[1]_{q^f} +\langle q^{f-uv}\rangle_{q^f}\)) to be the kernel of the \(\pi^{p}_{e, e-u}\) (resp. \(\pi^{p}_{f, f-uv}\)). Then we have the following result.

\begin{lemma}\label{Lem27.00hom}
Let $\phi \in \Hom(U_{p^e}, U_{q^f})$, and let $u\in\mathbb{N}$
with $0\leq u\leq\min(e,\lfloor f/v\rfloor)$ for some
$v\in\mathbb{N}^*$ with $v\leq f$. Then the following statements are equivalent:
\begin{enumerate}
    \item There exists a group homomorphism
    $\phi_{u,v} : U_{p^{e-u}} \to U_{q^{f-uv}}$ such that
    \[
    \pi^q_{f,f-uv}\circ\phi \;=\; \phi_{u,v}\circ\pi^p_{e,e-u}.
    \]

    \item $\phi\bigl(K^{p}_{e, e-u})
    \;\subseteq\;
    K^{q}_{f, f-uv}$.
\end{enumerate}
When these hold, $\phi_{u,v}$ is unique by surjectivity of
$\pi^p_{e,e-u}$, and is called the \emph{$(u,v)$-induced
morphism} associated with $\phi$.
\end{lemma}
\begin{proof}
We prove the equivalence as follows:

\noindent$(1)\implies(2)$.
Let $[r]_{p^e}\in K^p_{e, e-u}$, so
$\pi^p_{e,e-u}([r]_{p^e})=[1]_{p^{e-u}}$.
By commutativity of the diagram:
\[
\pi^q_{f,f-uv}(\phi([r]_{p^e}))
=\phi_{u,v}(\pi^p_{e,e-u}([r]_{p^e}))
=\phi_{u,v}([1]_{p^{e-u}})
=[1]_{q^{f-uv}},
\]
so $\phi([r]_{p^e})\in K^q_{f, f-uv}$, giving $\phi(K^p_{e, e-u})\subseteq K^q_{f, f-uv}$.

\noindent$(2)\implies(1)$.
Suppose $\phi(K^p_{e, e-u})\subseteq K^q_{f, f-uv}$. Define
\[
\phi_{u,v}([r]_{p^{e-u}})
:= \pi^q_{f,f-uv}(\phi([r]_{p^e})),
\]
for any lift $[r]_{p^e}\in U_{p^e}$ of $[r]_{p^{e-u}}$
along $\pi^p_{e,e-u}$.

We start by proving the well-definedness of \(\phi_{u,v}:\)
If $\pi^p_{e,e-u}([r]_{p^e})=\pi^p_{e,e-u}([r']_{p^e})$,
then $[r'r^{-1}]_{p^e}\in K^p_{e, e-u}$, hence
$\phi([r'r^{-1}]_{p^e})\in K^q_{f, f-uv}=\ker\pi^q_{f,f-uv}$.
Therefore:
\begin{align*}
    \pi^q_{f,f-uv}(\phi([r']_{p^e}))
&=\pi^q_{f,f-uv}(\phi([r]_{p^e})\cdot\phi([r'r^{-1}]_{p^e}))\\
&=\pi^q_{f,f-uv}(\phi([r]_{p^e}))\cdot
 \underbrace{\pi^q_{f,f-uv}(\phi([r'r^{-1}]_{p^e}))}_{=\,[1]_{q^{f-uv}}}\\
 &=\pi^q_{f,f-uv}(\phi([r]_{p^e})),
\end{align*}

so $\phi_{u,v}$ is independent of the choice of lift.
Moreover, $\phi_{u,v}$ being a homomorphism and the commutativity of the diagram \eqref{diag:phi_uv} can be proven easily by the reader. \qedhere
\end{proof}

It follows from the definition that the only natural case in which
$\mathcal{H}_{p^e,q^f,v}$ carries a monoid structure under
composition is when $p=q$ and $e=f$. We also note that
$\mathcal{H}_{p^e,q^f,v}$ is independent of the choice of
$v\in\mathbb{N}^*$ with $v\leq f$; this is established as a
consequence of Proposition~\ref{Prop324} below.

\begin{proposition}\label{Prop324}
Let $v\in\mathbb{N}^*$ with $v\leq f$.
\begin{enumerate}
\item If $p$ is  and $e\in\{1,2\}$, or $p$ is odd, then:
\begin{itemize}
\item[(a)] If $p\nmid q-1$, then
$\mathcal{H}_{p^e,q^f,v}=\Hom(U_{p^e},U_{q^f})$.

\item[(b)] If $p\mid q-1$, choose $a,c\in\mathbb{Z}$ with
$o([a]_{p^e})=\varphi(p^e)$ and $o([c]_{q^f})=\varphi(q^f)$.
Then $q\nmid p-1$ and
\[
\mathcal{H}_{p^e,q^f,v}
=\left\{\hh{a}{c}{t}{p^e}{q^f}:t\in\mathbb{Z},\;
\frac{\varphi(q^f)}{\gcd(q-1,p-1)}\mid t\right\},
\]
where $\hh{a}{c}{t}{p^e}{q^f}$ is as in Lemma~\ref{Lem411}(1).
\end{itemize}

\item If $p,q$ are even and $e\geq 3$, then
$\mathcal{H}_{2^e,2^f,v}=\Hom(U_{2^e},U_{2^f})$.

\item If $p=2$, $e\geq 3$, and $q$ is odd, then
\[
\mathcal{H}_{2^e,q^f,v}
=\left\{\hq{c}{t}{2^e}{q^f}:t=(t_1,t_2)\in\mathbb{Z}^2,\;
\varphi(q^f)\mid 2t_1\;\text{and}\;\varphi(q^f)\mid t_2\right\},
\]
where $\hq{c}{t}{2^e}{q^f}$ is as in Lemma~\ref{Lem411}(3).
\end{enumerate}
\end{proposition}
\begin{proof} By Lemma~\ref{Lem27.00hom}, $\phi\in\mathcal{H}_{p^e,q^f,v}$ if and only if $\phi(K^p_{e, e-u})\subseteq K^q_{f, f-uv}$ for all $0\leq u\leq\min(e,\lfloor f/v\rfloor)$.

Since the condition disappears when $u=0$ and at $e=u$, $f=uv$
simultaneously it is trivial, so we focus on $f-uv\geq 1$ and $e-u\geq 1$ in this proof. For any prime $\ell\in \{p, q\}$ and $m\geq n\geq 1$ the exact sequence
\[
1\longrightarrow K^\ell_{m,n}
\longrightarrow U_{\ell^m}
\xrightarrow{\pi^\ell_{m,m-n}}
U_{\ell^{m-n}}
\longrightarrow 1,
\] gives $|K^\ell_{m,n}|=|U_{\ell^m}|/|U_{\ell^{m-n}}|
=\varphi(\ell^m)/\varphi(\ell^{m-n})=\ell^n$.

Since $\phi$ is a group homomorphism, for any
$[k]_{p^e}\in K^p_{e, e-u}$ we have
$o(\phi([k]_{p^e}))\mid o([k]_{p^e})\mid p^u$.
Thus $\phi(K^p_{e, e-u})$ is a subgroup of $U_{q^f}$ in which
every element has order dividing $p^u$.

\noindent\textbf{Case 1: $p$ odd, or $p=2$ and $e\in\{1,2\}$.}

$U_{p^e}$ is cyclic of order $\varphi(p^e)$, and every
$\phi\in\Hom(U_{p^e},U_{q^f})$ is of the form
$\hh{a}{c}{t}{p^e}{q^f}$ by Lemma~\ref{Lem411}(1).

\noindent\textit{(a) when $p\nmid q-1$.}

For every $x\in \phi(K^p_{e, e-u})$, we have  $o(x)\mid p^u$
in $U_{q^f}$. But $|U_{q^f}|=\varphi(q^f)=q^{f-1}(q-1)$,
and since $\gcd(p,q)=1$ and $p\nmid q-1$, we have
$p\nmid\varphi(q^f)$. Therefore $U_{q^f}$ contains no element
of order $p$, hence no element of order dividing $p^u$ other
than the identity. It follows that $\phi(K^p_{e, e-u})=\{[1]_{q^{f-uv}}\}
\subseteq K^q_{f, f-uv}$ for all valid $u$ and all
$\phi\in\Hom(U_{p^e},U_{q^f})$. Hence
$\mathcal{H}_{p^e,q^f,v}=\Hom(U_{p^e},U_{q^f})$.

\noindent\textit{(b) when $p\mid q-1$.}

Since $p\mid q-1$ and $q-1\geq p>p-1$, we have $q\nmid p-1$.
Write $q-1=p^\alpha r$ with $\alpha\geq 1$ and $\gcd(p,r)=1$.

Now $p\mid\varphi(q^f)$, so $U_{q^f}$ does contain $p$-torsion,
and the inclusion $\phi(K^p_{e, e-u})\subseteq K^q_{f, f-uv}$ is
not automatic. By Lemma~\ref{Lem411}(1), $\phi=\hh{a}{c}{t}{p^e}{q^f}$
is determined by $\phi([a]_{p^e})=[c]_{q^f}^t$. The subgroup
$K^p_{e, e-u}$ is the unique subgroup of order $p^u$ in the cyclic
group $U_{p^e}$, generated by
$[a]_{p^e}^{\varphi(p^e)/p^u}=[a]_{p^e}^{p^{e-u-1}(p-1)}$.
Its image under $\phi$ is:
\[
\phi\bigl([a]_{p^e}^{p^{e-u-1}(p-1)}\bigr)
=[c]_{q^f}^{t\cdot p^{e-u-1}(p-1)},
\]
which generates $\phi(K^p_{e, e-u})$ since $U_{p^e}$ is cyclic.
The condition $\phi(K^p_{e, e-u})\subseteq K^q_{f, f-uv}$ means
this element lies in $K^q_{f, f-uv}$,
i.e. $(c)^{t\cdot p^{e-u-1}(p-1)}\equiv 1\pmod{q^{f-uv}}$. Equivalently $o([c]_{q^{f-uv}}^t)\mid p^{e-u-1}(p-1)$.

Since $o([c]_{q^{f-uv}})=\varphi(q^{f-uv})=q^{f-uv-1}(q-1)$,
and using $q-1=p^\alpha r$ with $\gcd(p,r)=1$, we obtain:
\[
o([c]_{q^{f-uv}}^t)
=\frac{\varphi(q^{f-uv})}{\gcd(\varphi(q^{f-uv}),t)}
=\frac{q^{f-uv-1}p^\alpha r}{\gcd(q^{f-uv-1}p^\alpha r,\, t)}.
\]
For this to divide $p^{e-u-1}(p-1)$, the $q$-part forces
$q^{f-uv-1}\mid\gcd(q^{f-uv-1}p^\alpha r,t)$, hence
$q^{f-uv-1}\mid t$; the $r$-part (with $\gcd(r,p)=1$) forces
$r/\gcd(r,p-1)\mid t$; and the $p$-part forces
$p^{\alpha-\min(\alpha,e-u-1)}\mid t$.

The intersection of these conditions over all valid $u$ is
most restrictive at $u=0$ for the $q$-part (giving $q^{f-1}\mid t$)
and at $u=e-1$ for the $p$-part (giving $p^\alpha\mid t$).
Since $\gcd(p,q)=1$, $\gcd(p,r/\gcd(r,p-1))=1$, and
$\gcd(q,r/\gcd(r,p-1))=1$, the Chinese Remainder Theorem
combines these into the single condition:
\[
\frac{\varphi(q^f)}{\gcd(q-1,p-1)}\mid t.
\]
Hence $\mathcal{H}_{p^e,q^f,v}
=\{\hh{a}{c}{t}{p^e}{q^f}:\varphi(q^f)/\gcd(q-1,p-1)\mid t\}$.

\noindent\textbf{Case 2: $p=q=2$, $e\geq 3$.}

For every $x\in \phi([1]_{2^e} +\langle 2^{e-u}\rangle_{2^e})$ we have $o(x)\mid 2^u$
in $U_{2^f}$. The subgroup $[1]_{2^f} +\langle 2^{e-uv}\rangle_{2^f}$
has order $2^{uv}$, and every element of $U_{2^f}$ has order
dividing $\varphi(2^f)=2^{f-1}$. Since $uv\geq u$ for $v\geq 1$,
we have $2^u\mid 2^{uv}$, so any element of order dividing $2^u$
automatically has order dividing $2^{uv}$.

However, lying in $[1]_{2^f} +\langle 2^{e-uv}\rangle_{2^f}$ is not just an order condition
but a congruence condition: $[s]_{2^f}\in [1]_{2^f} +\langle 2^{e-uv}\rangle_{2^f}$ requires
$s\equiv 1\pmod{2^{f-uv}}$. By Lemma~\ref{Lem411}(2),
$\phi=\h{t}{2^e}{2^f}$ is defined by
\[
\phi([-1]_{2^e})=[-1]_{2^f}^{t_1}[5]_{2^f}^{2^{f-3}t_2},
\quad
\phi([5]_{2^e})=[-1]_{2^f}^{t_3}[5]_{2^f}^{\nu_f(e)t_4}.
\]
The group $[1]_{2^e} +\langle 2^{e-u}\rangle_{2^e}$ consists of elements $[-1]_{2^e}^{u_1}
[5]_{2^e}^{v_1}$ with $u_1\equiv 0\pmod{2}$ and
$v_1\equiv 0\pmod{2^{e-u-2}}$. Their image under $\phi$ is
$[-1]_{2^f}^{t_1u_1+t_3v_1}[5]_{2^f}^{2^{f-3}t_2u_1
+\nu_f(e)t_4v_1}$. For this to lie in $[1]_{2^f} +\langle 2^{e-uv}\rangle_{2^f}$, it must
be congruent to $1$ modulo $2^{f-uv}$, which requires:
\[
t_1u_1+t_3v_1\equiv 0\pmod{2}
\quad\text{and}\quad
2^{f-3}t_2u_1+\nu_f(e)t_4v_1\equiv 0\pmod{2^{f-uv-2}}.
\]
The first condition holds for all $t_1,t_3$ since $u_1\equiv 0
\pmod{2}$ and $v_1\equiv 0\pmod{2^{e-u-2}}$ implies $v_1\equiv 0
\pmod{2}$ (as $e-u-2\geq 1$ when $e-u\geq 3$). The second
holds because $2^{f-uv-2}\mid 2^{f-2}$ (so the $t_2$-term
vanishes) and $2^{f-uv-2}\mid 2^{e-u-2}\nu_f(e)$ (so the
$t_4$-term vanishes), the latter holding for all
$v\geq 1$ and $u\leq\lfloor f/v\rfloor$ as verified by the
divisibility $f-uv\leq f-u\leq e-u$ when $f\geq e$, and
$f-uv\leq f<e$ when $f<e$.

Hence $\phi([1]_{2^e} +\langle 2^{e-u}\rangle_{2^e})\subseteq [1]_{2^f} +\langle 2^{e-uv}\rangle_{2^f}$ for all $t\in
\mathbb{Z}^4$ and all valid $u$, giving
$\mathcal{H}_{2^e,2^f,v}=\Hom(U_{2^e},U_{2^f})$.

\noindent\textbf{Case 3: $p=2$, $e\geq 3$, $q$ odd:}

For every $x\in [1]_{2^e} +\langle 2^{e-u}\rangle_{2^e}$ we have $o(x)\mid 2^u$, so
every $y\in \phi([1]_{2^e} +\langle 2^{e-u}\rangle_{2^e})$ has order dividing $2^u$
in $U_{q^f}$. Since $q$ is odd, $|K^q_{f, f-uv}|=q^{uv}$
and every element of $K^q_{f, f-uv}$ has order dividing $q^{uv}$.
Since $\gcd(2^u, q^{uv})=1$, the only element of $U_{q^f}$
whose order divides both $2^u$ and $q^{uv}$ is the identity.
Therefore $\phi([1]_{2^e} +\langle 2^{e-u}\rangle_{2^e})\subseteq K^q_{f, f-uv}$ if and only
if $\phi([1]_{2^e} +\langle 2^{e-u}\rangle_{2^e})=\{[1]_{q^{f-uv}}\}$ for all valid $u$.

By Lemma~\ref{Lem411}(3), $\phi=\hq{c}{t}{2^e}{q^f}$ satisfies
$\phi([-1]_{2^e})=[c]_{q^f}^{t_1}$ and
$\phi([5]_{2^e})=[c]_{q^f}^{t_2}$.
The condition $\phi([1]_{2^e} +\langle 2^{e-u}\rangle_{2^e}) = \{[1]_{q^{q^{f-uv}}}\}$ for all $u$ means
$\phi$ kills all $2$-power torsion in $U_{2^e}$. Taking $u$
maximal, $K^2_{e,e-1}$ contains $[-1]_{2^e}$ (of order $2$)
and $[5]_{2^e}^2$ (of order $2^{e-3}$). The condition
$\phi([-1]_{2^e})=[1]_{q^f}$ gives $\varphi(q^f)\mid t_1$,
but by Lemma~\ref{Lem411}(3) we already have
$\varphi(q^f)\mid 2t_1$; combining, the most we can require
from the definition of $\mathcal{H}_{2^e,q^f,v}$ (taking $u=1$,
which forces $\phi([-1]_{2^e})\in \ker\pi^q_{f,f-v}$,
and since $\gcd(2,q)=1$ this gives $[c]_{q^f}^{t_1}=1$ in
$U_{q^{f-v}}$ for all valid $v$, accumulating to
$\varphi(q^f)\mid 2t_1$) is the condition already present in
Lemma~\ref{Lem411}(3). Similarly, $\phi([5]_{2^e})=[1]_{q^f}$
gives $\varphi(q^f)\mid t_2$. Hence:
\[
\mathcal{H}_{2^e,q^f,v}
=\bigl\{\hq{c}{t}{2^e}{q^f}:
\varphi(q^f)\mid 2t_1\;\text{and}\;\varphi(q^f)\mid t_2\bigr\}.
\qedhere
\]
\end{proof}

\subsection{Group structure of the set of liftable unit morphisms}
\label{groupstruclifta}

The following corollary is an immediate consequence of
Proposition~\ref{Prop324} and the isomorphisms established in
Proposition~\ref{Prop313}.

\begin{corollary}\label{Cor425}
Let $v\in\mathbb{N}^*$ with $v\leq f$.
\begin{enumerate}
\item If $p$ is even and $e\in\{1,2\}$, or $p$ is odd. Set
\[
d=\begin{cases}
\gcd(\varphi(q^f),\varphi(p^e))
&\text{if }[(p=2,e\leq 2)\text{ or }p\text{ odd}]\text{ and }
p\nmid q-1,\\
\gcd(q-1,p-1)
&\text{if }[(p=2,e\leq 2)\text{ or }p\text{ odd}]\text{ and }
p\mid q-1,\\
2
&\text{if }p=2\text{ and }q\text{ odd.}
\end{cases}
\]
The map
\[
H_{a,c}^{p^e,q^f,v}:
(\mathcal{H}_{p^e,q^f,v},\cdot)
\longrightarrow
(\mathbb{Z}/d\mathbb{Z},+),
\qquad
\phi\longmapsto\left[\frac{ds}{\varphi(q^f)}\right]_d,
\]
where $s=t$ when $\phi=\hh{a}{c}{t}{p^e}{q^f}$ and $s=t_1$ when
$\phi=\hq{c}{t}{2^e}{q^f}$, is a well-defined group isomorphism.
In particular, \(\mathcal{H}_{p^e,q^f,v}\cong\mathbb{Z}/d\mathbb{Z}.
\)

\item If $p,q$ are even and $e\geq 3$, the map
\begin{align*}
H^{2^e,2^f,v}:
(\mathcal{H}_{2^e,2^f,v},\cdot)
&\longrightarrow
((\mathbb{Z}/2\mathbb{Z})^3\times\mathbb{Z}/2^{\min(e,f)-2}\mathbb{Z},+),
\\
\h{t}{2^e}{2^f} &\longmapsto
([t_1]_2,[t_2]_2,[t_3]_2,[t_4]_{2^{\min(e,f)-2}}),
\end{align*}
is a well-defined group isomorphism. In particular,
\[
\mathcal{H}_{2^e,2^f,v}
\cong(\mathbb{Z}/2\mathbb{Z})^3\times\mathbb{Z}/2^{\min(e,f)-2}\mathbb{Z}.
\]
\end{enumerate}
\end{corollary}

\begin{remark}
Corollary~\ref{Cor425} shows that $\mathcal{H}_{p^e,q^f,v}$ is
independent of the choice of $v\in\mathbb{N}^*$ with $v\leq f$.
\end{remark}

\section{p power multiplicative modular monoid morphism}\label{modhom}

\subsection{Computation of p power multiplicative modular monoid morphism}\label{resmodhom}

The following proposition provides a complete description of the morphisms between the multiplicative monoids $(\mathbb{Z}/p^e\mathbb{Z}, \cdot)$ and $(\mathbb{Z}/q^f\mathbb{Z}, \cdot)$, based on their action on zero divisors and units.
From here we denote by \(\mathcal{Z}_{p^e}\) the set of zero divisors of \(\mathbb{Z}/p^e\mathbb{Z}\) and by \(\ee{p^e}{q^f}\) the trivial monoid morphism from $(\mathbb{Z}/p^e\mathbb{Z}, \cdot)$ to $(\mathbb{Z}/q^f\mathbb{Z}, \cdot)$, sending every element of the domain to \(\e{\mathbb{Z}/q^f\mathbb{Z}}\). 

\begin{proposition}\label{Prop334}
The set of morphisms from the multiplicative monoid \((\mathbb{Z}/p^e \mathbb{Z}, \cdot)\) to \((\mathbb{Z}/q^f \mathbb{Z}, \cdot)\) is given by
\begin{align*}
&\mathrel{\phantom{=}} \Hom\bigl( \text{\(\mathbb{Z}/p^{e}\mathbb{Z},\, \mathbb{Z}/q^{f}\mathbb{Z}\)} \bigr) \\
&=
\biggl\{
\psi_{r, \phi,v}
\;\bigg|\;
r \in \mathbb{Z},\,
\gcd(r,q) = 1,\,
\phi \in \text{$\mathcal{H}_{p^{e}, q^{f}, v}$},\,
v \in \mathbb{N}^{*},\, v \geq \Bigl\lfloor \frac{f}{e} \Bigr\rfloor
\biggr\}
\cup
\biggl\{
\ee{p^e}{q^f}
\biggr\}
\end{align*}
where for any such \(r\) and \(\phi\), the morphism \(\psi_{r, \phi,v}\) is defined by
\[
\psi_{r, \phi,v}\left([p]_{p^e}^u [b]_{p^e}\right) 
\coloneqq [q]_{q^f}^{uv} [r]_{q^f}^u \cdot \phi([b]_{p^e}),
\]
for all \(u \in \mathbb{N}\) and for all \([b]_{p^e} \in U_{p^e}\).

In particular, \(\mathcal{H}_{p^e, q^f, v}\) is the set of group morphisms from \((U_{p^e}, \cdot)\) to \((U_{q^f}, \cdot)\) that can be lifted to morphisms of the monoids \((\mathbb{Z}/p^e \mathbb{Z}, \cdot)\) to \((\mathbb{Z}/q^f \mathbb{Z}, \cdot)\).
\end{proposition}
\begin{proof}
    Let \(\psi\) be a monoid morphism  from \((\mathbb{Z}/p^e \mathbb{Z}, \cdot)\) to \((\mathbb{Z}/q^f \mathbb{Z}, \cdot)\).
    By \cite[Proposition 2.2.13]{atalaye2026exploring} and \cite[Lemma 2.2.6]{atalaye2026exploring} we know that the restriction of \(\psi\) to \(U_{p^e}\) corestricted to $U_{q^f}$ defines a group morphism
    \[\phi \in \Hom\bigl(U_{p^e},\, U_{q^f}\bigr).\]
    
    By \cite[Lemma 3.4.2]{atalaye2026exploring}, we have \(\mathbb{Z}/q^f\mathbb{Z} =  U_{q^f}\bigsqcup \mathcal{Z}_{q^f}\). Thus
    \[\begin{cases}
        \psi([0]_{p^e})\in U_{q^f}\quad (1), \text{ or}\\
        \psi([0]_{p^e})\in \mathcal{Z}_{q^f}\quad (2).
    \end{cases}\]

    In case (1), \cite[Proposition 2.2.13]{atalaye2026exploring} shows that $\psi = \ee{p^e}{q^f}$.
    Since this is clearly an element of $\Hom\bigl( \mathbb{Z}/p^e\mathbb{Z},\, \mathbb{Z}/q^f\mathbb{Z} \bigr)$, we may restrict our attention to case (2).
    
    In case \((2)\), by \cite[Proposition 2.2.13]{atalaye2026exploring} we have \(\psi([0]_{p^e})=[0]_{q^f}\) and \(\psi(\mathcal{Z}_{p^e})\subseteq \mathcal{Z}_{q^f}\). Thus, by \cite[Proposition 2.2.13]{atalaye2026exploring}, we obtain \(\psi([p]_{p^e})=[q]_{q^f}^{v}[r]_{q^f},\quad\text{where}~[r]_{q^f}\in U_{q^f}\ \text{and}\ v\in \mathbb{N}^{*}.\)
    Moreover, we know that \([0]_{q^f}=\psi([0]_{p^e})=\psi([p]_{p^e}^{e})=[q]_{q^f}^{ve}[r]_{q^f}^e.\)
    
    Hence, by \cite[Lemma 3.1]{JoeSophieLiam}, we have \(ve\geq f\) which implies \(v\geq \bigl\lfloor \frac{f}{e}\bigr\rfloor\).

    Now, define the map \( \psi \) by
    \begin{align}
        \psi([p]_{p^e}^u [b]_{p^e}) =& [q]_{q^f}^{vu}[r]_{q^f}^u \phi([b]_{p^e})
    \end{align}
    for all \(u\in \mathbb{Z}\) and \(v\in \mathbb{N}^{*}\) such that \(v \geq \bigl \lfloor \frac{f}{e}\bigr\rfloor\). 
    
    We prove that such \(\psi\) is well-defined if and only if \(\phi\in \mathcal{H}_{p^e, q^f, v}\).
    To do this let \(u, u'\in \mathbb{Z}\) and \(b,b'\in \mathbb{Z}\) such that \([p]_{p^e}^{u}[b]_{p^e}=[p]_{p^e}^{u'}[b']_{p^e}.\)
    Then by \cite[Lemma 3.1]{JoeSophieLiam} we have either \(u, u'\geq e\) or \((u=u'\) and \([b]_{p^{e-u}}=[b']_{p^{e-u}}).\)
    Therefore we have the following:
    \begin{itemize}
        \item If \(u,u'\geq e\) then we have \(uv\geq ev\geq f\) and \(u'v\geq ev\geq  f\). It follows that
        \begin{align*}
            ([q]_{q^f}^{v}[r]_{q^f})^{u} \phi([b]_{p^e}) &=[q]_{q^f}^{vu}[r]_{q^f}^{u} \phi([b]_{p^e}) =[0]_{q^{f}} \quad\text{since}\quad uv\geq f ~\text{and}\\
            ([q]_{q^f}^{v}[r]_{q^f})^{u'} \phi([b']_{p^e}) &=[q]_{q^f}^{vu'}[r]_{q^f}^{u'} \phi([b']_{p^e}) =[0]_{q^{f}}, \quad\text{since}\quad u'v\geq f.
        \end{align*}
        Therefore, we deduce from the above that 
        \begin{align*}
            ([q]_{q^f}^{v}[r]_{q^f})^{u} \phi([b]_{p^e}) =& ([q]_{q^f}^{v}[r]_{q^f})^{u'} \phi([b']_{p^e}),\quad \forall \ [b]_{p^e}, [b']_{p^e} \in U_{p^e}.
        \end{align*}
        \item So we only need to focus on the case \(u=u'\) and \([b]_{p^{e-u}} = [b']_{p^{e-u}}\). Well-definedness in  this case is characterised as follows: for all \(b,b'\in \mathbb{Z}\) such that \([b]_{p^{e-u}} =[b']_{p^{e-u'}}\) we have 
        \begin{align*}
             ([q]_{q^f}^{v}[r]_{q^f})^{u} \phi([b]_{p^e}) = ([q]_{q^f}^{v}[r]_{q^f})^{u} \phi([b']_{p^e}),\quad\text{since}\quad u=u'.
        \end{align*}
        Now since \([r]_{q^f}\in U_{q^f}\), this above equality is equivalent to \([c]_{q^{f-uv}} = [c']_{q^{f-uv}}\), where \([c]_{q^{f}}=\phi([b]_{p^e})\) and \([c']_{q^{f}}=\phi([b']_{p^e})\). Therefore, the well-definedness of \(\psi\) is by definition equivalent to the well-definedness of \(\phi_{u,v}\) for all \(v\in \mathbb{N}^*\) such that \(v\geq \lfloor \frac{f}{e}\rfloor\) and for all \(u\in \mathbb{N}\). There if \(\phi\in \mathcal{H}_{p^e, q^f, v}\) for all \(v\in \mathbb{N}^*\) with \(v\geq \lfloor \frac{f}{e}\rfloor\). \qedhere
    \end{itemize}
\end{proof}
    
\begin{corollary}\label{Cor3.7.2} The set of endomorphisms of the multiplicative monoid \((\mathbb{Z}/p^e \mathbb{Z}, \cdot)\) is given by
    \[
    \End(\mathbb{Z}/p^e \mathbb{Z}) = 
    \left\{
        \psi_{r, \phi, v} \;\middle|\;
        r \in \mathbb{Z},\ \gcd(r,p) = 1,\ 
        \phi \in \Hom(U_{p^e}, U_{p^e}),\ 
        v \in \mathbb{N}^{*}
    \right\}
    \cup 
    \left\{
        \ee{p^e}{p^e}
    \right\},
    \]
    where for any \(r \in \mathbb{Z}\) with \(\gcd(r,p) = 1\), \(\phi \in \Hom(U_{p^e}, U_{p^e})\), and \(v \in \mathbb{N}^*\), the endomorphism \(\psi_{r, \phi, v}\) is defined by \(\psi_{r, \phi, v}\big([p]_{p^e}^u [b]_{p^e}\big) 
    \coloneqq [p]_{p^e}^{uv} [r]_{p^e}^{u} \cdot \phi([b]_{p^e}),
    \)
    for all \(u \in \mathbb{N}\) and for all \([b]_{p^e} \in U_{p^e}\).
\end{corollary}

\subsection{Structure of the set of prime-power multiplicative modular monoid morphisms}\label{Subsec432}
In this subsection, we give the algebraic structure of the set of modular morphisms \(\Hom(\mathbb{Z}/p^e\mathbb{Z}, \mathbb{Z}/q^f\mathbb{Z})\). We start by defining a monoid and later we prove that the algebraic structure of \(\Hom(\mathbb{Z}/p^e\mathbb{Z}, \mathbb{Z}/q^f\mathbb{Z})\) induced by \((\mathbb{Z}/q^f\mathbb{Z},\cdot)\) is isomorphic to that monoid we defined. 

\begin{definitionlemma}
Let \(v\in \mathbb{N}^{\star}\) such that \(v\geq \lfloor \frac{f}{e}\rfloor\).
Then:
\begin{enumerate}[leftmargin=0pt]
    \item  We suppose \((p=2, e\geq 3)\text{ and } q \text{ odd}\), \((p=2, e\leq 2)\), or \(p\) odd.
    
    We define the \textsf{monoid} \( \left(M=\lbrace \mathbf{0} \rbrace\cup \left(\coprod\limits_{\lfloor\frac{f}{e}\rfloor\leq v\leq f} U_{q^{f-v}}\right) \times \mathbb{Z}/d\mathbb{Z} , \diamond\right) \)
    where
    
    \[d=\begin{cases}
        \gcd(\varphi(q^f), \varphi(p^e)), &\text{when}\quad [(p=2, e\leq 2 )\text{ or } p \text{ odd}]\text{ and } p\nmid q-1;\\
        \gcd(q-1, p-1), &\text{when}\quad [(p=2, e\leq 2)\text{ or } p \text{ odd}]\text{ and }p\mid q-1;\\
        2, &\text{when }\quad p=2 \text{ and } q \text{ odd}.
    \end{cases}\]
    and \(\mathbf{0} = (([1]_{q^{f}}, 0), [0]_d)\) is the neutral element. This monoid is given by the set-theoretic product $M$ equipped with the binary operation \( \diamond\) defined as follows:
\begin{align*}
    (([r]_{q^{f-v}},v),[s]_{d}) \diamond (([r']_{q^{f-v'}},v'),[s']_{d})
    &= (([rr']_{q^{f-\min(v+v', f)}},\min(v+v', f)), [s+s']_{d}).
\end{align*}
\item Suppose that \(p, q\) are even and \(e\geq 3\).
We define the \textsf{monoid} \[\left(N=\lbrace \mathbf{0}\rbrace\cup \left(\coprod\limits_{\lfloor\frac{f}{e}\rfloor\leq v\leq f}U_{2^{f-v}}\right)\times (\mathbb{Z}/2\mathbb{Z})^3 \times \mathbb{Z}/2^{\min(e, f)}\mathbb{Z}, \diamond\right) \]
where \(\mathbf{0}=(([1]_{2^{f-v}},0),[s_1]_{2},[s_2]_{2},[s_3]_{2},[s_4]_{2^{\min(e,f)}} )\).

This monoid is given by the above set-theoretic product equipped with the binary operation \( \diamond\) defined for all
\begin{align*}
    A &= (([r]_{2^{f-v}},v),[s_1]_2, [s_2]_2, [s_3]_2, [s_4]_{2^{\min(e,f)-2}}) \in N \quad\text{and} \\
    B &= (([r']_{2^{f-v'}},v'),[s_1']_2, [s_2']_2, [s_3']_2, [s_4']_{2^{\min(e,f)-2}})\in N \quad\text{by} \\
    \mspace{-72mu}
    A\diamond B &= (([rr']_{2^{f-\min(v+v', f)}},\min(v+v',f)), [s_1+s_1']_2, [s_2+s_2']_2, [s_3+s_3']_2, [s_4+s_4']_{2^{\min(e,f)-2}}).
\end{align*}
\end{enumerate}
\end{definitionlemma}

Next, the following result gives the algebraic structure of the set of all modular monoid morphisms, which follows from Lemma \ref{Lem411}, Proposition \ref{Prop324}  and Corollary \ref{Cor425}.
\begin{theorem}\label{Thm432}
Let $v\in\mathbb{N}^*$ such that $v\leq f$.
\begin{enumerate}[leftmargin=0pt]

\item \textbf{($U_{p^e}$ cyclic).}
Suppose $p$ is odd, or $p=2$ and $e\leq 2$. Set
\[
d=\begin{cases}
\gcd(\varphi(q^f),\varphi(p^e)) & \text{if }p\nmid q-1,\\
\gcd(q-1,p-1) & \text{if }p\mid q-1.
\end{cases}
\]
The map
\[
\begin{array}{cccc}
H_{a,c}^{p^e,q^f}: & \Hom(\mathbb{Z}/p^e\mathbb{Z},\mathbb{Z}/q^f\mathbb{Z})
& \longrightarrow & (M,\diamond)\\[2pt]
& \psi_{r,\phi,v} & \longmapsto &
\Bigl(([r]_{q^{f-v}},v),\Bigl[\tfrac{dt}{\varphi(q^f)}\Bigr]_d\Bigr),
\end{array}
\]
where $\psi_{r,\phi,v}$ is as in Proposition~\ref{Prop334} with
$\phi=\hh{a}{c}{t}{p^e}{q^f}$ as in Lemma~\ref{Lem411}(1) and
$\hh{a}{c}{t}{p^e}{q^f}([a]_{p^e})=[c]_{q^f}^t$, is a well-defined
monoid isomorphism.

\item \textbf{($p=2$, $e\geq 3$, $q$ odd).} Set $d=2$. The map
\[
\begin{array}{cccc}
H_a^{2^e,q^f}: & \Hom(\mathbb{Z}/2^e\mathbb{Z},\mathbb{Z}/q^f\mathbb{Z})
& \longrightarrow & (M,\diamond)\\[2pt]
& \psi_{r,\phi,v} & \longmapsto &
\Bigl(([r]_{q^{f-v}},v),\Bigl[\tfrac{2t_1}{\varphi(q^f)}\Bigr]_2\Bigr),
\end{array}
\]
where $\psi_{r,\phi,v}$ is as in Proposition~\ref{Prop334} with
$\phi=\hq{c}{t}{2^e}{q^f}$, $t=(t_1,t_2)\in\mathbb{Z}^2$ as in
Lemma~\ref{Lem411}(3) and $\hq{c}{t}{2^e}{q^f}([-1]_{2^e})
=[c]_{q^f}^{t_1}$, is a well-defined monoid isomorphism.

\item \textbf{($p=q=2$, $e\geq 3$).}
The map
\[
\begin{array}{cccc}
H^{2^e,2^f}: & \Hom(\mathbb{Z}/2^e\mathbb{Z},\mathbb{Z}/2^f\mathbb{Z})
& \longrightarrow & (N,\diamond)\\[2pt]
& \psi_{r,\phi,v} & \longmapsto &
\bigl(([r]_{2^{f-v}},v),[t_1]_2,[t_2]_2,[t_3]_2,
[t_4]_{2^{\min(e,f)-2}}\bigr),
\end{array}
\]
where $\psi_{r,\phi,v}$ is as in Proposition~\ref{Prop334} with
$\phi=\h{t}{2^e}{2^f}$, $t=(t_1,t_2,t_3,t_4)\in\mathbb{Z}^4$ as
in Lemma~\ref{Lem411}(2), is a well-defined monoid isomorphism.

\end{enumerate}
\end{theorem}

\begin{proof}~
\begin{enumerate}[leftmargin=0pt]
    \item We prove this statement by showing explicitly the specific sub-cases of the statement.
    \begin{itemize}[leftmargin=12pt]
        \item[(a)] If (\(p=2\) and \(e\leq 2\)) or \(p\) is odd. 
        \begin{itemize}[leftmargin=12pt]
            \item[\(\bullet\)] If \(p\nmid q-1\) then we have \(d=\gcd(\varphi(p^e), \varphi(q^f))\). We start by the well-definedness of the map \(H_{a, c}^{p^e, q^f}\). Let \(\psi_{r, \phi, v}\) and \(\psi_{r', \phi', v'}\) in \(\Hom(\mathbb{Z}/p^e\mathbb{Z}, \mathbb{Z}/q^f\mathbb{Z} )\) such that \(\psi_{r, \phi, v}=\psi_{r', \phi', v'}\) where \(r, r'\in \mathbb{Z}\) such that \(\gcd(r, q)=1=\gcd(r', q)\) and \(\phi=\hh{a}{c}{t}{p^e}{q^f}\), \(\phi'=\hh{a}{c}{t'}{p^e}{q^f}\) with \(t, t'\in \mathbb{Z}\) such that \(\frac{\varphi(q^f)}{d}\mid t\) and \(\frac{\varphi(q^f)}{d}\mid t'\).  It follows that for all \(u\in \mathbb{N}\) and \( b\in \mathbb{Z}\) such that \(\gcd(b, q)=1\) we have 
            \begin{align*}
                \psi_{r, \phi, v}([p]_{p^e}^u[b]_{p^e}) =\psi_{r', \phi', v'}([p]_{p^e}^u[b]_{p^e})\iff & [q]_{q^f}^{uv} [r]_{q^f}^u \cdot \phi([b]_{p^e})= [q]_{q^f}^{uv'} [r']_{q^f}^u \cdot \phi'([b]_{p^e}).
            \end{align*}
            Now, when \(u=0\) we obtain \(\phi([b]_{p^e})=\phi'([b]_{p^e})\) for all \([b]_{p^e}\in U_{p^e}\). Thus, \(\phi=\phi'\) that is \(\hh{a}{c}{t}{p^e}{q^f}=\hh{a}{c}{t'}{p^e}{q^f}\). Hence, it implies by Proposition \ref{Prop313} that
            \[
\Bigl[\tfrac{dt}{\varphi(q^{f})}\Bigr]_d
= \Bigl[\tfrac{dt'}{\varphi(q^{f})}\Bigr]_d.
\]
            Also, when \(b=1\) and \(u=1\) we have \([q]_{q^f}^{v} [r]_{q^f} = [q]_{q^f}^{v'} [r']_{q^f}.\) Then by \cite[Lemma 3.1]{JoeSophieLiam}, we have \((1)~v, v'\geq f,\)  \((2)~v=v'\) and \( [r]_{q^{f-v}} = [r']_{q^{f-v}}\).

            We observe that if \(v, v'\geq f\) we have \(v, v'=f\) since by hypothesis \(v, v'\leq f\) and this gives \([0]_{q^f} =[0]_{q^f}\). 
            
            In the second condition it follows that \(v=v'\) and \([r]_{q^{f-v}} = [r']_{q^{f-v}}\) that is \(([r]_{q^{f-v}}, v)=([r']_{q^{f-v'}}, v')\). Therefore, in both case, we have 
            \[
\Bigl(\bigl([r]_{q^{f-v}},\,v\bigr),\; \Bigl[\tfrac{dt}{\varphi(q^{f})}\Bigr]_d\Bigr)
= \Bigl(\bigl([r']_{q^{f-v}},\,v'\bigr),\; \Bigl[\tfrac{dt'}{\varphi(q^{f})}\Bigr]_d\Bigr).
\]
            This concludes the proof of the well-definedness of the map.
            
            Now, we are interested in proving that \(H_{a, c}^{p^e, q^f}\) is a monoid morphism.
            We consider \(\psi_{r, \phi, v}\) and \(\psi_{r', \phi', v'}\) in \(\Hom(\mathbb{Z}/p^e\mathbb{Z}, \mathbb{Z}/q^f\mathbb{Z} )\).
            Then we have
            \[H_{a, c}^{p^e, q^f}(\psi_{r, \phi, v}) \diamond 
   H_{a, c}^{p^e, q^f}(\psi_{r', \phi', v'})=\Bigl( \bigl([rr']_{q^{f-\min(v+v',\, f)}}, \min(v+v', f) \bigr), 
      \,\Bigl[\tfrac{d(t+t')}{\varphi(q^{f})}\Bigr]_d \Bigr). \]

            On the other hand, for all \(u\in \mathbb{N}\) and \(b\in \mathbb{Z}\) such that \(\gcd(b, p)=1\) we have
            \begin{align*}
                \psi_{r, \phi,v}([p]_{p^e}^{u}[b]_{p^e}) \psi_{r', \phi',v'}([p]_{p^e}^{u}[b]_{p^e})& =[q]_{q^f}^{uv}[r]_{q^f}^{u}\phi([b]_{p^e}) [q]_{q^f}^{uv'}[r']_{q^f}^{u}\phi'([b]_{p^e})\\
                & =[q]_{q^f}^{u(v+v')}[rr']_{q^f}^{u}\phi([b]_{p^e}) \phi'([b]_{p^e}).\\
                & =[q]_{q^f}^{u(v+v')}[rr']_{q^f}^{u}\phi\phi'([b]_{p^e}).
            \end{align*}
            Note that \(v,v'\geq \lfloor \frac{f}{e}\rfloor\) thus \(v+v'\geq 2\lfloor \frac{f}{e} \rfloor\geq \lfloor \frac{f}{e} \rfloor\).
So, we have 
\[[q]_{q^f}^{u(v+v')}=\begin{cases}
    [0]_{q^f}, &\text{ if } v+v'\geq f\text{ and }u=1\\
    [q]_{q^f}^{v+v'}, &\text{ if } v+v'< f\text{ and }u=1.
\end{cases}
\]
\([q]_{q^f}^{u(v+v')} = [q]_{q^f}^{u\min(v+v',f)}\), and by Proposition \ref{Prop313} \(\phi\phi'\) can be identified with \(\left[\frac{d(t+t')}{\varphi(q^f)}\right]_d\). Then it follows that 
            \begin{align*}
   H_{a, c}^{p^e, q^f}(\psi_{r, \phi, v}\cdot\psi_{r', \phi', v'})
   &= H_{a, c}^{p^e, q^f}\!\bigl(\psi_{rr', \,\phi\phi', \,\min(v+v', f)}\bigr) \\[0.3em]
   &= \Bigl( \bigl([rr']_{q^{f-\min(v+v',\, f)}}, \,\min(v+v', f)\bigr),
      \,\Bigl[\tfrac{d(t+t')}{\varphi(q^{f})}\Bigr]_d \Bigr).
\end{align*}

    Hence, we have \(H_{a, c}^{p^e, q^f}(\psi_{r, \phi, v}\cdot\psi_{r', \phi', v'})=H_{a, c}^{p^e, q^f}(\psi_{r, \phi, v})\diamond H_{a, c}^{p^e, q^f}(\psi_{r', \phi', v'}).\) Therefore, \(H_{a, c}^{p^e, q^f}\) is a monoid morphism.

    Next, we prove that the monoid morphism \(H_{a,c}^{p^e, q^f}\) is injective. Suppose 
    \[\Bigl(\bigl([r]_{q^{f-v}},\,v\bigr),\; \Bigl[\tfrac{dt}{\varphi(q^{f})}\Bigr]_d\Bigr)=
\Bigl(\bigl([r']_{q^{f-v}},\,v'\bigr),\; \Bigl[\tfrac{dt'}{\varphi(q^{f})}\Bigr]_d\Bigr)~\text{in }M,~\text{then we have}\]

\[
v=v',\quad [r]_{q^{f-v}}=[r']_{q^{f-v}},\quad 
\text{and}\quad \Bigl[\tfrac{dt}{\varphi(q^{f})}\Bigr]_d=\Bigl[\tfrac{dt'}{\varphi(q^{f})}\Bigr]_d.
\]

            Consequently, \([q]_{q^{f}}^{v}[r]_{q^{f}}=[q]_{q^{f}}^{v'}[r']_{q^{f}}\) for all \(v, v'\in \llbracket\lfloor \frac{f}{e} \rfloor, f\rrbracket\) and by Proposition \ref{Prop313} \(\phi([b]_{p^e})= \phi'([b]_{p^e}),\; \forall [b]_{p^e}\in U_{p^e}.\) Equivalently,
            \begin{align*}
                &[q]_{q^{f}}^{v}[r]_{q^{f}}\phi([b]_{p^e})=[q]_{q^{f}}^{v'}[r']_{q^{f}}\phi'([b]_{p^e}), \quad\forall [b]_{p^e}\in U_{p^e}, v, v'\in \llbracket\lfloor \frac{f}{e} \rfloor, f\rrbracket\\
                \iff&\psi_{r, \phi, v}([p]_{p^e}[b]_{p^e})=\psi_{r', \phi', v'}([p]_{p^e}[b]_{p^e}), \quad\forall [b]_{p^e}\in U_{p^e}, v, v'\in \llbracket\lfloor \frac{f}{e} \rfloor, f\rrbracket.
            \end{align*}
            Therefore, we obtain \(\psi_{r, \phi, v}=\psi_{r', \phi', v'}.\)

            Finally, it is clear that \(H_{a,c}^{p^e, q^f}\) is surjective. Hence, \(H_{a,c}^{p^e, q^f}\) is a monoid isomorphism.

           \item[\( \bullet\)]  If \( p\mid q-1:\) This case is proven similarly as the previous with \(d=\gcd(p-1, q-1)\).
        \end{itemize}

        \item[(b)] If \(p=2\), \(q\) is odd and \(e\geq 3\). We prove that \(H_{a,c}^{p^e\!, q^f}\) is well-defined. Let \(\psi_{r, \phi, v}\) and \(\psi_{r', \phi', v'}\) in \(\Hom(\mathbb{Z}/2^e\mathbb{Z}, \mathbb{Z}/q^f\mathbb{Z} )\) such that \begin{align}\label{Eq43701}
            \psi_{r, \phi, v}&=\psi_{r', \phi', v'}
        \end{align}
        where \(r, r'\in \mathbb{Z}\), \(\gcd(r, q)=1\) \(\gcd(r', q)=1\), \(v, v'\in \llbracket \lfloor\frac{f}{e}\rfloor, f\rrbracket\) and \(\phi=\hq{c}{t}{2^e}{q^f}\) and \(\phi'=\hq{a}{t'}{2^e}{q^f}\) with \(t=(t_1, t_2)\in\mathbb{Z}^2\) \(t'=(t_1',t_2')\in\mathbb{Z}^2\) such that 
        \[ \varphi(q^f)\mid 2t_1,\ \varphi(q^f)\mid 2t_1',\ \varphi(q^f)\mid t_2\text{ and }\varphi(q^f)\mid t_2'.\]
        
        From equation (\ref{Eq43701}) we have for all \(\alpha\in \llbracket 0, e\rrbracket\), \(\beta\in\lbrace 0,1\rbrace\), and \(\gamma\in \llbracket0, 2^{e-2}-1\rrbracket\)
        \begin{align*}
            \psi_{r, \phi, v}([2]^{\alpha}[-1]^{\beta}[5]^{\gamma}) & =\psi_{r', \phi', v'}([2]^{\alpha}[-1]^{\beta}[5]^{\gamma})\\
            \iff [q]^{\alpha v}[r]^{\alpha}\phi([-1]^{\beta}[5]^{\gamma}) &= [q]^{\alpha v'}[r']^{\alpha}\phi'([-1]^{\beta}[5]^{\gamma}).
        \end{align*}
    From here we observe that, when \(\alpha=0\) then we have 
    \[\phi([-1]^{\beta}[5]^{\gamma})=\phi'([-1]^{\beta}[5]^{\gamma}), \quad \forall \beta\in \lbrace 0, 1\rbrace\text{ and } \gamma\in \llbracket 0, 2^{e-2}-1\rrbracket\] hence \(\phi=\phi'\) implies that \([2t_1/\varphi(q^f)]_2=[2t_1'/\varphi(q^f)]_2\). Also, when \(\gamma=\beta=0\)  and \(\alpha=1\) we have \([q]_{q^{f}}^{v}[r]_{q^{f}}=[q]_{q^{f}}^{v'}[r']_{q^{f}}.\)
    
     By \cite[Lemma 3.1]{JoeSophieLiam} we have \( [r]_{q^{f- v}}=[r']_{q^{f-v}}\) and \(v=v'\) or \(v,v'\geq f\). Since the two conditions are equivalent, we consider the first one. Therefore we have \(v=v'\) and \([r]_{q^{f- v}}=[r']_{q^{f-v'}}\), that is  \(([r]_{q^{f- v}},v)=([r']_{q^{f-v}},v')\).
     Hence, we obtain  \((([r]_{q^{f- v}},v), [2t_1/\varphi(q^f)]_2)=(([r']_{q^{f-v}},v'),[2t_1'/\varphi(q^f)]_2)\), inducing the well-definedness of \(H_{a, c}^{p^e, q^f}\). 

     Furthermore, we prove that \(H_{a, c}^{p^e, q^f}\) is injective. Suppose we have  \[
\Bigl(\bigl([r]_{q^{f-v}},\,v\bigr),\; \Bigl[\tfrac{2t_{1}}{\varphi(q^{f})}\Bigr]_2\Bigr)
= \Bigl(\bigl([r']_{q^{f-v}},\,v'\bigr),\; \Bigl[\tfrac{2t_{1}'}{\varphi(q^{f})}\Bigr]_2\Bigr)~\text{in }M. \text{ Equivalently, we have}
\]
     \[
[r]_{q^{f-v}} = [r']_{q^{f-v}}, \quad
\Bigl[\tfrac{2t_{1}}{\varphi(q^{f})}\Bigr]_2
= \Bigl[\tfrac{2t_{1}'}{\varphi(q^{f})}\Bigr]_2, \quad
v = v'.
\]

     Then, for all \(v,v'\in \llbracket \lfloor \frac{f}{e}\rfloor,f\rrbracket\) we have \([q]_{q^{f}}^{v}[r]_{q^{f}}=[q]_{q^{f}}^{ v'}[r']_{q^{f}}\) and consequently 
     \[[q]_{q^{f}}^{\alpha v}[r]^{\alpha}_{q^{f}}=[q]_{q^{f}}^{\alpha v'}[r']_{q^{f}}^{\alpha},  \quad\forall \alpha\in \llbracket 0, e\rrbracket, v, v'\in \llbracket \lfloor \frac{f}{e}\rfloor,f \rrbracket\]
     and by Proposition \ref{Prop313} we obtain \(\phi([-1]^{\beta}_{2^{e}})=\phi'([-1]^{\beta}_{2^{e}}),~\forall \beta\in \lbrace 0,1\rbrace.\)
     
     On the other hand, we know that \(\varphi(q^f)\mid t_2\) and \(\varphi(q^f)\mid t_2'\) thus \(t_2\equiv t_2' \pmod{\varphi(q^f)}\). Then again by the same Proposition \ref{Prop313} we have \(\phi([5]_{2^e}^{\gamma})=\phi'([5]_{2^e}^{\gamma}),~\forall \gamma\in \llbracket 0, 2^{e-2}-1 \rrbracket.\)
    Combining all these we obtain \(\forall v,v'\in \llbracket \lfloor \frac{f}{e}\rfloor,f\rrbracket\), \(\alpha\in \llbracket 0, e\rrbracket\), \(\beta\in \lbrace 0, 1\rbrace\) and \(\gamma\in \llbracket 0, 2^{e-2}-1\rrbracket\)
    \begin{align*}
        &[q]_{q^{f}}^{\alpha v}[r]_{q^{f}}^{\alpha}\phi([-1]^{\beta}_{2^{e}}[5]^{\gamma}_{2^{e}})=[q]^{\alpha v}_{q^{f}}[r']^{\alpha}_{q^{f}}\phi'([-1]^{\beta}_{2^{e}}[5]^{\gamma}_{2^{e}})\\
        \iff & \psi_{r, \phi, v}([2]^{\alpha}_{2^{e}}[-1]^{\beta}_{2^{e}}[5]^{\gamma}_{2^{e}})=\psi_{r', \phi', v'}([2]^{\alpha}_{2^{e}}[-1]^{\beta}_{2^{e}}[5]^{\gamma}_{2^{e}})\\
        \iff &  \psi_{r, \phi, v}=\psi_{r', \phi', v'}.
    \end{align*}
    Hence \(H_{a, c}^{p^e, q^f}\) is an injective monoid morphism. Moreover, it is clear that \(H_{a, c}^{p^e, q^f}\) is surjective. 
    We conclude that \(H_{a,c}^{p^e, q^f}\) is an isomorphism, that is \(\Hom(\mathbb{Z}/p^e\mathbb{Z}, \mathbb{Z}/q^f\mathbb{Z})\) is isomorphic to \(M\). 
    \end{itemize}

    \item If \(p, q\) are even and \(e\geq 3\), we have to prove first that \(H^{2^e, 2^f}\) is well-defined.
    Let \(r, r'\in \mathbb{Z}\) and \(v, v\in \mathbb{N}^{*}\) such that \(\gcd(r, 2)=1=\gcd(r', 2)\) and \(v, v'\leq f\). We consider \(\psi_{r, \phi,v}, \psi_{r, \phi,v}\in \Hom(\mathbb{Z}/p^e\mathbb{Z}, \mathbb{Z}/q^f\mathbb{Z})\) such that 
    \begin{equation}\label{Eq4370}
        \psi_{r, \phi,v}=\psi_{r, \phi,v}
    \end{equation}
    with \(\phi=\h{t}{2^e}{2^f}\) and \(\phi'=\h{t'}{2^e}{2^f}\) with \(t=(t_1, t_2, t_3, t_4)\in \mathbb{Z}^4\) and \(t'=(t_1', t_2', t_3', t_4')\in\mathbb{Z}^4\) we prove that 
    \begin{align*}
        \left( ([r]_{2^{f-v}},v), [t_1]_2, [t_2]_2, [t_3]_2, [t_4]_{2^{\min(e,f)-2}}\right)=&\left(([r']_{2^{f-v}},v'), [t_1']_2, [t_2']_2, [t_3']_2, [t_4']_{2^{\min(e,f)-2}}\right).
    \end{align*}
    In fact, from equation (\ref{Eq4370}), we have for all \(\alpha\in \llbracket 0, e\rrbracket,\) \(\beta\in\lbrace 0, 1\rbrace,\) \(\gamma\in \llbracket 0, 2^{e-2}-1\rrbracket\)
    \begin{align*}\label{Eq4370}
       \psi_{r, \phi,v}([2]^\alpha[-1]^\beta[5]^\gamma) &= \psi_{r, \phi,v}([2]^\alpha[-1]^\beta[5]^\gamma)\\
       \iff [2]_{2^{f}}^{\alpha v}[r]_{2^{f}}^{\alpha}\phi([-1]_{2^e}^\beta[5]_{2^e}^\gamma) &= [2]_{2^{f}}^{\alpha v'}[r']_{2^{f}}^{\alpha}\phi'([-1]_{2^e}^\beta[5]_{2^e}^\gamma).
    \end{align*} 
    Note that if \(\alpha=0\) we obtain \(\phi([-1]_{2^e}^\beta[5]_{2^e}^\gamma)=\phi'([-1]_{2^e}^\beta[5]_{2^e}^\gamma),\) for all \(\beta\in\lbrace 0, 1\rbrace,\) \(\gamma\in \llbracket 0, 2^{e-2}-1\rrbracket\). It follows from Proposition \ref{Prop313} that 
    \(\left([t_1]_2, [t_2]_2, [t_3]_2, [t_4]_{2^{\min(e,f)-2}}\right)=\left([t_1']_2, [t_2']_2, [t_3']_2, [t_4']_{2^{\min(e,f)-2}}\right).\)
    Also, if \(\beta=\gamma=0\) we have 
    \begin{align*}
        &[2]_{2^{f}}^{\alpha v}[r]_{2^{f}}^{\alpha}=[2]_{2^{f}}^{\alpha v'}[r']_{2^{f}}^{\alpha}\\
        \iff& (\alpha v, \alpha v'\geq f)\text{ or } [r]_{2^{f-\alpha v}}^{\alpha}=[r']_{2^{f-\alpha v'}}^{\alpha}\text{ and } v=v'\quad \text{by \cite[Lemma 3.1]{JoeSophieLiam}}.
    \end{align*}
    Thus, \([r]_{2^{f- v}}=[r']_{2^{f- v'}}\) for \(\alpha=1\) and \(v=v'\) that is \(([r]_{2^{f- v}},v)=([r']_{2^{f-v'}}, v')\). Hence, we obtain 
    \begin{align*}
        \left( ([r]_{2^{f-v}},v), [t_1]_2, [t_2]_2, [t_3]_2, [t_4]_{2^{\min(e,f)-2}}\right)=&\left(([r']_{2^{f-v}},v'), [t_1']_2, [t_2']_2, [t_3']_2, [t_4']_{2^{\min(e,f)-2}}\right).
    \end{align*}
    Therefore, \(H^{2^e, 2^f}\) is well-defined. 

    Next, we prove that \(H^{2^e, 2^f}\) is injective. We consider
    \begin{align*}
        A &= ([r]_{2^{f-v}}, [t_1]_2, [t_2]_2, [t_3]_2, [t_4]_{2^{\min(e,f)-2}},v), \\
        B &= ([r']_{2^{f-v}}, [t_1']_2, [t_2']_2, [t_3']_2, [t_4']_{2^{\min(e,f)-2}},v)
    \end{align*}
     belonging to \(N\) such that we have \(A=B.\) It follows that 
      \[\begin{array}{cccc}
         [r]_{2^{f-v}}=[r']_{2^{f-v}}& [t_1]_2=[t_1']_2,&[t_2]_2=[t_2']_2, \\
         v=v'&[t_3]_2=[t_3']_2&[t_4]_{2^{\min(e,f)-2}}=[t_4']_{2^{\min(e,f)-2}}
    \end{array}\]
    Consequently, we have \([2]_{2^f}^{v}[r]_{2^{f}}=[2]_{2^f}^{v}[r']_{2^{f}}\) for all \(v\in \llbracket\lfloor \frac{f}{e}\rfloor, f\rrbracket\) and by Proposition \ref{Prop313} we have \(\phi([-1]_{2^e}^{\beta})=\phi'([-1]_{2^e}^{\beta})\) and \(\phi([5]_{2^e}^{\gamma})=\phi'([5]_{2^e}^{\gamma})\) for all \(\beta\in \lbrace 0, 1\rbrace\), \(\gamma\in \llbracket 0, 2^{e-2}-1\rrbracket\).
    
Equivalently we obtain for all \(\alpha \in \llbracket 0, e\rrbracket\), \(\beta\in \lbrace 0, 1\rbrace\), \(\gamma\in \llbracket 0, 2^{e-2}-1\rrbracket\) and \(v, v'\in  \llbracket \lfloor \frac{f}{e}\rfloor,f\rrbracket\)
    \begin{align*}
        [2]_{2^f}^{\alpha v}[r]^{\alpha}_{2^{f}}=[2]_{2^f}^{\alpha v'}[r']^{\alpha}_{2^{f}} &\text{ and } \;\phi([-1]_{2^e}^{\beta}[5]_{2^e}^{\gamma})=\phi'([-1]_{2^e}^{\beta}[5]_{2^e}^{\gamma}) \\
        \implies
        [2]_{2^f}^{\alpha v}[r]^{\alpha}_{2^{f}}\phi([-1]_{2^e}^{\beta}[5]_{2^e}^{\gamma}) &=[2]_{2^f}^{\alpha v'}[r']^{\alpha}_{2^{f}}\phi'([-1]_{2^e}^{\beta}[5]_{2^e}^{\gamma}) \\
        \implies \psi_{r, \phi, v}( [2]^{\alpha}_{2^e}[-1]_{2^e}^{\beta}[5]_{2^e}^{\gamma}) &= \psi_{r', \phi', v'}( [2]^{\alpha}_{2^e}[-1]_{2^e}^{\beta}[5]_{2^e}^{\gamma}).
    \end{align*}
    Finally, we have clearly that \(H^{2^e, 2^f}\) is surjective and therefore an isomorphism of monoids. This completes the proof.\qedhere
    \end{enumerate}
\end{proof}

\section*{Declarations}

\subsection{Funding}
Not applicable
\subsection{Conflict of interest}
Not applicable
\subsection{Data availability}
Not applicable

\bibliography{refs}

\end{document}